\documentclass[a4paper, 11pt]{amsart}

\usepackage{amsthm, amsmath, amssymb} 
\usepackage{graphicx}                 
\usepackage{dsfont}                   
\usepackage[bf]{caption}              
\usepackage[english]{babel}
\usepackage[ansinew]{inputenc}
\usepackage{color}
\usepackage{underscore}
\usepackage[all]{xy}
 \usepackage{pb-diagram,pb-xy}

\newtheorem{thm}{Theorem}[section]
\newtheorem{cor}[thm]{Corollary}
\newtheorem{lem}[thm]{Lemma}
\newtheorem{prop}[thm]{Proposition}

\theoremstyle{definition}
\newtheorem{defi}[thm]{Definition}
\newtheorem{ex}[thm]{Example}

\newtheoremstyle{rmk}
  {12pt}                   
  {12pt}                   
  {}                       
  {}                       
  {\normalfont\bfseries}   
  {.}                      
  {\newline}               
  {}

\theoremstyle{rmk}

\newtheorem{rmk}[thm]{Remark}
\newtheorem{rmk/defi}[thm]{Remark/Definition}

\newcommand{\N} { \mathbb{N} }  

\newcommand{\R} { \mathbb{R} }

\newcommand{\holim}{\textup{holim}}
\newcommand{\hocolim}{\textup{hocolim}}
\newcommand{\op}{\textup{op}}
\newcommand{\hodim}{\textup{hodim}}
\newcommand{\Tot}{\textup{Tot}}

\newcommand{\holimsub}[1]{\begin{array}[t]{cc} \textup{holim} \\ [-1mm]
\scriptstyle{#1} \end{array}}

\newcommand{\hocolimsub}[1]{\begin{array}[t]{cc} \textup{hocolim} \\
[-1.2mm] \scriptstyle{#1} \end{array}}

\begin{document}\parindent 0pt 

\title{Manifold Calculus adapted for Simplicial Complexes}

\vspace{0,3cm}

\author{Steffen Tillmann}
\address{Math.~Institut, Universit\"at M\"unster, 48149 M\"unster, Einsteinstrasse 62, Germany}
\email{s_till05@uni-muenster.de}

\begin{abstract} 
We develop a generalization of manifold calculus in the sense of Goodwillie-Weiss where the manifold is replaced by a simplicial complex. We consider functors from the category of open subsets of a fixed simplical complex into the category of topological spaces and prove an analogue of the approximation theorem. Namely, under certain conditions such a functor can be approximated by a tower of (appropriately adapted) polynomial functors.
\end{abstract}
\maketitle

\tableofcontents\setcounter{tocdepth}{3}
\addtocounter{section}{-1}

\section{Introduction}

Let $K$ be a simplicial complex\footnote[1]{By simplicial complex we mean the geometric realization of a simplical complex.} and let $\mathcal{O}(K)$ be the category of open subsets of $K$ and inclusions between open subsets. 
Then we consider 
contravariant functors $F$ from $\mathcal{O}(K)$ to the category of topological spaces $(Top)$. Such a functor $F$ is called \textit {good } if it takes stratified isotopy equivalences to weak equivalences and if it fulfils the (co)limit axiom. Roughly speaking, a stratified isotopy equivalence is a simplexwise smooth isotopy equivalence (for a precise Definition see \ref{isotopyequ}). This notion emphasizes the important property of the simplicial complex that each stratum carries a smooth structure, but note that when $K$ comes from a smooth triangulation of a smooth manifold, the notion of stratified isotopy equivalence does not agree with the usual notion of isotopy equivalence. \\

We will define the Taylor approximations $T_{k}F$ of a good contravariant functor $F$ and show that they are appropriate approximations to $F$ under certain additional conditions. In order to define these functors $T_{k}F$, we have to introduce a category of special open subsets $\mathcal{O}k(K)$ depending on an integer $k\geq 0$.
We define $\mathcal{O}k(K)$ to be a full subcategory of $\mathcal{O}(K)$. The objects are those open subsets of $K$ with at most $k$ connected components where each component is stratified isotopy equivalent to an open star neighbourhood in $K$. Then $T_{k}F:\mathcal{O}(K)\rightarrow (Top)$ is given by
\begin{align}\label{intr}
T_{k}F(V) := \holimsub{U\in\mathcal{O}k(K), \hspace{1mm}U\subset V} F(U)
\end{align}
By analogy with manifold calculus, we can define \textit{$k$-polynomial} functors. One of the main results of this paper is that the functors $T_{k}F$ are $k$-polynomial (Corollary \ref{kpolyn}). Every $k$-polynomial functor has the property that it is determined by its restriction to the 
 subcategory $\mathcal{O}k(K)$ of $\mathcal{O}(K)$ (see Theorem \ref{F1F2}). \\

We have canonical natural transformations $F\rightarrow T_{k}F$ and restriction transformations $T_{k}F\rightarrow T_{k-1}F$ induced by the inclusions $\mathcal{O}(k-1)(K)\hookrightarrow \mathcal{O}k(K)$ for all $k\in\N$. This gives us a tower of functors - by analogy we call it the \textit{Taylor tower} - and a canonical natural transformation
\[
F \rightarrow T_{\infty}F := \holim_{k}\hspace{1mm} T_{k}F
\]
By definition, the Taylor tower converges to $F$ if this map is a weak equivalence for every $V \in \mathcal O(K)$. We want to define conditions under which the tower converges. Therefore, we introduce the notion of a \textit{$\rho$-analytic} functor where  $\rho >0$ is an integer (see Definition \ref{nalytic}). Morally, $\rho$ is the radius of convergence. The following theorem is the main result of this paper. \\

\begin{thm}\label{maintheorem}
Let $\rho > \text{dim}(K)$ be an integer. If the functor $F$ is good and $\rho$-analytic, the canonical map 
\[
F(V)\rightarrow T_{\infty}F(V) = \holimsub{U\in \cup_{k} \mathcal{O}k(K), U\subset V} F(U)
\]
is a weak equivalence for all $V\in\mathcal{O}(K)$.
\end{thm}
\vspace{0,2cm}

The analogue of Theorem \ref{maintheorem} in the setting of Goodwillie-Weiss 
is proven by induction on the (relative) handle index of a compact, smooth codimension zero submanifold of $M$. In order to find an appropriate analogue of the handle index, we have to introduce a compact codimension zero subobject in a simplicial complex.
To this end, we use the smooth structure of each (open) simplex. So roughly speaking, we define a codimension zero subobject as well as its handle index simplexwise. 
In particular, we get a handle index function which assigns to each simplex its handle index. 
The handle index of a codimension zero subobject in a simplicial complex is then defined as the maximum of this function over its simplices. We will show that this notion is different from its analogue in a smooth manifold. For example, if $M$ is a compact smooth manifold, then ($M$ is a compact codimension zero subobject of itself and) in general the handle index of $M$ in the usual sense is not equal to the handle index of a fixed triangulation whereby $M$ is regarded as a simplicial complex. \\
As our main application we study occupants in simplicial complexes \cite{occupantsinsimplicialcomplexes}. Let $M$ be a smooth manifold and $K\subset M$ be a simplicial complex where each closed simplex is smoothly embedded in $M$. We look for a homotopical formula for $M\setminus K$ in terms of spaces $M\setminus T$ where $T$ is a finite subset of $K$. The finite subset $T\subset K$ could be regarded as a finite set of occupants. In the smooth setting, where $K$ is replaced by a smooth submanifold $L \subset M$, this has been done in \cite{occupants}.
But by studying occupants in simplicial complexes we are allowed to consider more general situations. This also leads directly to generalizations of results in \cite{confcat} and \cite{pontryjagin}. For more details of this application see Chapter \ref{occsection}. As another example we study spaces of stratified smooth embeddings from a simplicial complex $K$ into a smooth manifold $M$.   \\

Can we compare this new theory with the Goodwillie-Weiss manifold calculus? We have seen that some of the key definitions are very different. We need to consider stratified isotopy equivalences, redefine $\mathcal{O}k$ and adapt the definition of a  codimension zero subobject. So it may come as a surprise that this new version is a generalization of Goodwillie-Weiss manifold calculus in the following sense: Let $M$ be a smooth manifold and $F$ be a good functor from the category of open subsets of $M$ to the category of topological spaces. 
We can choose a triangulation of $M$. Then the $k$-th stage of the Goodwillie-Weiss tower as defined in \cite {mancal1} coincides (up to homotopy) with $T_{k}F$ as defined in (\ref{intr}) (see Theorem \ref{vergltt}). \\

Can we generalise the results here to more general situations?
In fact, the methods developed and results proved in this paper could also be adapted for stratified manifolds more generally. At the moment we have no applications of such an extended theory. So, for simplicity, we restrict ourselves to the case of simplicial complexes which is adequate for the applications we have in mind. \\

\textbf{Notation:} The category $(Top)$ is the category of topological spaces. By a simplex $S$ of a simplicial complex, we mean a nondegenerate closed simplex. For such a simplex $S$, we denote by $\op(S)$ the open simplex. For a positive integer $k$, we set $[k]:=\left\{0,1,...,k\right\}$.  \\

\textbf{Acknowledgment:} This paper is a part of the author's PhD thesis under the supervision of Michael Weiss. It is a pleasure to thank him for suggesting this interesting topic, supporting the author and improving an earlier draft of this version.

\section{Polynomial functors}

We start to adapt the basic definitions. We introduce good and $k$-polynomial functors as well as the category $\mathcal{O}k$ of special open subsets and study the relationship between them. To this end, we will also introduce a concept of handle index in a simplicial complex.

\subsection{Basic definitions}

Let $K$ be a simplicial complex. We define the category $\mathcal{O}=\mathcal{O}(K)$ as follows: The objects are the open subsets of $K$ and the morphisms are inclusions, i.e. for $U,V\in\mathcal{O}$ there is exactly one morphism $U\rightarrow V$ if $U\subset V$ and there are no morphisms otherwise. 
\begin{defi}
Let $U,V\in\mathcal{O}$ be open subsets and let $f_{0},f_{1}: U \rightarrow V$ be two maps such that $f_{i}|_{U\cap S}$ is a smooth embedding from $U\cap S$ into $V\cap S$ for all simplices $S$ of $K$ and $i=0,1$. We call $f_{0}$ and $f_{1}$ \textit{stratified isotopic} if there is a continuous map
$H:U\times \left[ 0,1 \right] \rightarrow V$ such that 
\[
H|_{(U\cap S)\times \left[ 0,1 \right]}: (U\cap S)\times \left[ 0,1 \right]\rightarrow (V\cap S)
\] 
is a smooth isotopy from $f_{0}|_{U\cap S}$ to $f_{1}|_{U\cap S}$ for all simplices $S$ of $K$. \\
Note: For an $n$-dimensional simplex $S$, we regard $U\cap S$ as a subspace in the euclidean space $\R^{n+1}$. 
\end{defi}
\begin{defi}\label{isotopyequ}
Let $U,V\in\mathcal{O}$ be two open subsets with $U\subset V$. The inclusion $i: U\rightarrow V$ is a \textit{stratified isotopy equivalence} if there is a map $e: V\rightarrow U$ such that 
$e|_{V\cap S}$ is an embedding from $V\cap S$ into $U\cap S$ for all simplices $S$ of $K$ and $i\circ e$, respectively $e\circ i$, is stratified isotopic to $id_{V}$, 
respectively $id_{U}$.
\end{defi}
\begin{defi}\label{good}
A contravariant functor $F: \mathcal{O}\rightarrow (Top)$ is \textit{good} if
\begin{enumerate}
\item $F$ takes stratified isotopy equivalences to weak homotopy equivalences
\item for every family $\left\{ V_{i} \right\}_{i\in \N}$ of objects in $\mathcal{O}$ with $V_{i}\subset V_{i+1}$ for all $i\in \N$, the following canonical map is a weak homotopy equivalence:
\begin{center}
$F(\cup_{i} V_{i}) \rightarrow \holim_{i}\hspace{1mm} F(V_{i})$
\end{center}
\end{enumerate}
\end{defi}

Recall: For a positive integer $k$, let $\mathcal{P}([k])$ be the power set of $[k]$. Then a functor from $\mathcal{P}([k])$ to $(Top)$ is a $k$-cube of spaces. 
\begin{defi}
Let $\chi$ be a cube of spaces. The \textit{total homotopy fiber} of $\chi$ is the homotopy fiber of the canonical map
\begin{align*}
\chi(\emptyset ) \rightarrow \holimsub{\emptyset\neq T\subset [k]} \hspace{0,15cm}\chi (T)
\end{align*}
If this map is a weak homotopy equivalence, we call the cube $\chi$ \textit{(weak homotopy) cartesian}.
\end{defi}
Now we are going to define polynomial functors. Therefore let $F$ be a good functor, let $V\in \mathcal{O}$ be an open subset of $K$, and let $A_{0},A_{1},...,A_{k}$ be pairwise disjoint closed subsets of $V$ (for a positive integer $k$). Define a $k$-cube by
\begin{align}\label{polcube}
T\mapsto F(V\setminus \cup_{i\in T} A_{i})
\end{align}
\begin{defi}\label{polynomial}
The functor $F$ is \textit{polynomial of degree} $\leq k$ if the $k$-cube defined in (\ref{polcube}) is cartesian for all $V\in\mathcal{O}$ and pairwise disjoint closed subsets $A_{0},A_{1},...,A_{k}$ of $V$.
\end{defi}
\begin{prop}
Let $F:\mathcal{O}\rightarrow (Top)$ be a good contravariant functor which is polynomial of degree $\leq k$. Then $F$ is also polynomial of degree $\leq k+1$. 
\end{prop}
\begin{proof}
Let $V\in\mathcal{O}$ be an open subset and let $A_{0},A_{1},...,A_{k+1}$ be pairwise disjoint closed subsets of $V$. We have to show that the canonical map
\[
F(V) \rightarrow \holimsub{\emptyset \neq T \subset [k+1]} F(V\setminus A_{T})
\]
is a weak equivalence where $A_{T}:= \cup_{i\in T} A_{i}$. This is equivalent (see section 1 of \cite{goodwillie2}) to saying that the following commutative diagram is a homotopy pullback:
\begin{equation*}
\begin{gathered}
\xymatrix{
F(V) \ar[d] \ar[r] & \holim_{\emptyset\neq T\subset[k]} \hspace{1mm}F(V\setminus A_{T}) \ar[d] \\
F(V\setminus A_{k+1}) \ar[r] & \text{holim}_{\emptyset\neq T\subset[k]} \hspace{1mm}F(V\setminus (A_{T}\cup A_{k+1}))
}
\end{gathered}
\end{equation*}
By assumption, the horizontal arrows are weak equivalences. Therefore, the diagram is a homotopy pullback.
\end{proof}
Manifold calculus assigns a Taylor tower to each good contravariant functor (see \cite{mancal1}). More precisely: For a good functor $F$ there is a $k$-polynomial functor $T_{k}F$ for all $k$ which coincides with $F$ on a full subcategory of special open sets (depending on $k$). Our aim is to construct an analogous theory for simplicial complexes. To this end, we need the notation of a special open set. \\

Let $x\in K$ be given and let $\mathcal{S}_{x}$ be the open star of the open simplex containing $x$, i.e. $\mathcal{S}_{x} :=\cup_{S} \hspace{0,1cm} \op(S)$ where the union ranges over all closed simplices $S$ of $K$ such that $x$ is an element of $S$. 

\begin{defi}
For a positive integer $k$, we define a full subcategory $\mathcal{O}k(K) = \mathcal{O}k$ of $\mathcal{O}$. Its objects are the open subsets $V\subset K$ with the following properties: $V$ has at most $k$ connected components and for each component $V_{0}$ of $V$, there is an $x\in K$ such that $V_{0}\subset \mathcal{S}_{x}$ and the inclusion $V_{0}\rightarrow\mathcal{S}_{x}$ is a stratified isotopy equivalence. An element of $\mathcal{O}k$ (for some $k$) is called a \textit{special} \textit{open set}.
\end{defi}
\begin{rmk}\label{compI1}
By definition, up to stratified isotopy equivalence the category $\mathcal{O}1$ has as many objects as the simplicial complex $K$ has simplices.
\end{rmk}
We will work out the relationship between the category $\mathcal{O}k$ and polynomial functors of degree $\leq k$.

\subsection{Handle index in a simplicial complex}\label{handleif}

For a compact manifold, there is a concept of relative handle index (see \cite{mancal2}). 
Reminder: Given a manifold triad $Q$, there are boundary sets $\partial_{0}Q$ and $\partial_{1}Q$ and a corner set $\partial_{0}Q\cap\partial_{1}Q$. The \textit{relative handle index} of $Q$ is the smallest integer $q$ such that $Q$ can built from a collar on $\partial_{0}Q$ by attaching handles of index $\leq q$. If $Q$ is a collar on $\partial_{0}Q$, then the handle index is $-\infty$.

\begin{ex}
(1) Let $D^{n}:=\left\{ x\in \R^{n} \mid \left\|x\right\| \leq 1\right\}$ be the $n$-disk. Then, $Q:= D^{q}\times D^{j-q}$ is a manifold triad with boundary sets $\partial_{0}Q:=S^{q-1}\times D^{j-q}$ and $\partial_{1}Q:=D^{q}\times S^{n-q-1}$. The relative handle index is $q$. \\
(2) Let $M$ be a smooth manifold with boundary and $f:M\rightarrow \R$ be a smooth map such that $0$ is a regular value for $f$ and $f|_{\partial M}$. Then $Q:= f^{-1} \left(\left[0,\infty \right)\right)$ is a manifold triad with $\partial_{0} Q = \partial M\cap Q$. Every $Q \subset M$ which can be obtained in this way will be called \textit{codimension zero subobject in a manifold} (compare \cite[\S 0]{mancal2}).
\end{ex}

We need an analogous concept of codimension zero subobjects in simplicial complexes:

\begin{defi}
A subset $P \subset K$ is called a \textit{codimension zero subobject} if there is a map $f: K \rightarrow \R$ such that 
\begin{enumerate}
    \item $f|_{S} : S \rightarrow \R$ is smooth for all simplices $S$ of $K$
		\item $P:= f^{-1} \left(\left[0,\infty \right)\right)$
		\item for all simplices $S$ of $K$: $0$ is a regular value for $f|_{\op(S)}$
\end{enumerate}
Note that for every simplex $S$, $P \cap S$ is a manifold triad (in a non-smooth sense) with $\partial_{0} (P\cap S) = \partial S \cap P$.
\end{defi}

\begin{defi}\label{tame}
An open subset $V\in\mathcal{O}$ of $K$ is called \textsl{tame} if it is the interior of a codimension zero compact subobject $C$ of $K$.
\end{defi}

\textbf{Notation}: Let $K^{n}\subset K$ be the $n$-skeleton of $K$, i.e. $K^{n}$ is the union of all $m$-simplices of $K$ with $m\leq n$. For a subset $U\in K$ we set $U^{n}:= U\cap K^{n}$.

\begin{rmk}
Let $V\in\mathcal{O}$ be tame. Then $V$ satisfies the following condition: For all simplices $S_{u}$ and all subsimplices $S_{v}\subset S_{u}$, we have 
\[
\text{cl}(V\cap S_{v})=\text{cl}(V\cap S_{u})\cap S_{v}
\]
where for a subset $U$ of $K$, $\text{cl}(U)$ is the closure of $U$ in $K$.
This statement emphasizes an important property of tame open subsets. In particular, the set $\op(S)\subset K$ where $S$ is a simplex of $K$ need not be tame in $K$, even if it is open in $K$. 
\end{rmk}

Now we define the \textit{handle index function} $f_{V}:\N\rightarrow \N\cup\left\{ -\infty \right\}$ for a tame set $V\in\mathcal{O}$. By definition, $V$ is the interior of a compact codimension zero subobject $C$ of $K$. Define $C_{u}:=S_{u}\cap C$ for all simplices $S_{u}$ of $K$ and let $I$ be the finite set of all $u$ with $C_{u}\neq \emptyset$. Note: Every $C_{u}$ is a manifold triad. 

In more detail: Let $u\in I$ be given and let $S_{u}$ be an $n$-simplex. A closed simplex is a manifold with boundary. Therefore, $C_{u}$ is a compact manifold with corners. The boundary sets are given by $\partial_{0}C_{u}=\partial S_{u} \cap C_{u}$ and $\partial_{1}C_{u}$ is the closure of $\partial C_{u} \cap \op(S_{u})$ in $C_{u}$. Therefore, the corner set is given by $\partial_{0}C_{u}\cap\partial_{1}C_{u}= \partial (C_{u}\cap\partial S_{u})$. 

Choose a handle decomposition for $C_{u}$ relative to $\partial_{0} C_{u}$ and let $q_{u}$ be the handle index of $C_{u}$ relative to $\partial_{0} C_{u}$. Note:
\begin{align*}
\partial_{0} C_{u} = \partial S_{u} \cap C_{u} = K^{n-1} \cap C_{u} = C_{u}^{n-1}
\end{align*} 

\begin{defi}
We set $f_{V}(j):= \text{max}_{u\in I(j)} \hspace{0,15cm}q_{u}$ where $I(j)\subset I$ is the subset of all $u\in I$ such that $S_{u}$ is a $j$-simplex. If $I(j)=\emptyset$, we set $f_{V}(j):= -\infty$. \\
The function $f_{V}:\N\rightarrow \N\cup\left\{ -\infty \right\}$ is called the \textit{handle index function} of $V$ and the integer $q_{V}:= \text{max}_{j\in\N}\hspace{0,1cm} f_{V}(j)$ is called the \textit{handle index} of $V$.
\end{defi}

\begin{ex}
Let $K$ be an $1$-dimensional simplicial complex with four $0$-simplices $S_{0},S_{1},S_{2},S_{3}$ and three $1$-simplices $I_{1},I_{2},I_{3}$ which are defined by $I_{k}:=\left\{ S_{k-1},S_{k} \right\}$ for $k=1,2,3$. Then $K$ is identified with the interval $\left[ 0,3 \right]$ by the identifications $S_{l} = l$ for $l=0,1,2,3$ and $I_{k}=\left[k-1,k\right]$ for $k=1,2,3$. Let $V\in\mathcal{O}$ be a tame open set and $f_{V}$ be the handle index function. By definition, we have $f_{V}(j)=-\infty$ for all $j\geq 2$. \\
Let $V:= \left[ 0; 0.5 \right) \in\mathcal{O}$. The handle index function of $V$ is then given by $f_{V}(0)=0$ and $f_{V}(1)= -\infty$ because $V$ is a collar of the $0$-simplex $S_{0}$. \\
For $V:= \left( 1.2 ; 1.8 \right)$ we have $f_{V}(0)= -\infty$ and $f_{V}(1)=0$ because $V\cap K^{0} = \emptyset$. \\
Now we consider a more interesting example. Up to now we only considered special open sets, i.e. elements of $\mathcal{O}k$ for some $k$. Now we define the open set $W:= \left( 0.5 ; 2.5 \right)$ so that $W$ is not a special open set. Then the handle index function is given by $f_{W}(0)=0$, $f_{W}(1)=1$.
\end{ex}

\begin{ex}
Let $K$ be an $n$-simplex. Then $K\in\mathcal{O}$ is a tame open set. Therefore, we can consider the handle index function $f_{K}$. It is defined by $f_{K}(j)=j$ for all $0\leq j\leq n$ and $f_{K}(j)= -\infty$ for all $j>n$.
\end{ex}

\subsection{Polynomial functors and special open sets}

In manifold calculus it is shown that a polynomial functor is determined by its restriction to a selection of special open sets \cite[Theorem 5.1]{mancal1}. We can verify an analogous result by extending the proof of \cite[5.1]{mancal1}. Therefore, we need the following concept of a collar.

\begin{rmk/defi}\label{col}
Let $V\in\mathcal{O}$ be a tame open set such that there is an integer $j\leq \text{dim}(K)$ with $f_{V}(m)\leq 0$ for all $m>j$, let $S'$ be a $j$-simplex of $K$ and let $A\subset \op(S')$ be compact in the open simplex $\op(S')$. By definition of the handle index function, there is a closed subset $\text{col}_{V}(A)$ in $V$ - \textit{the collar of $A$ in $V$} - such that there are diffeomorphisms
\[ 
\text{col}_{V}(A)\cap S \cong A \times \left[0,1\right)^{n-j}
\]
for each $n$-simplex $S$ with $S'\subset S$, compatibly as $S$ runs through the simplices of $K$ with $S'\subset S$. What does \textit{compatibly} mean? If $S_{1}$ is a $n_{1}$-simplex and $S_{2}$ is a $n_{2}$-simplex of $K$ with $S'\subset S_{1}\subset S_{2}$, then the following diagram commutes
\begin{equation*}
\begin{gathered}
\xymatrix{
\text{col}_{V}(A)\cap S_{1} \ar[d]^{incl.} \ar[r]^{\cong} & A \times \left[0,1\right)^{n_{1}-j} \ar[d] \\
\text{col}_{V}(A)\cap S_{2} \ar[r]^{\cong} & A \times \left[0,1\right)^{n_{2}-j}
}
\end{gathered}
\end{equation*}
where the right vertical arrow is the canonical inclusion, in particular it is the identity in the first coordinate.  \\
Note that we constructed the collar $\text{col}_{V}(A)$ of $A$ in $V$ uniquely up to stratified isotopy equivalence. 
\end{rmk/defi}

\begin{thm}\label{F1F2}
Let $F_{1}\rightarrow F_{2}$ be a natural transformation between $k$-polynomial functors. If $F_{1}(V)\rightarrow F_{2}(V)$ is a weak equivalence for all $V\in\mathcal{O}k$, it is a weak equivalence for all $V\in\mathcal{O}$. 
\end{thm}

\begin{proof}
Using the (co)limit axiom (the second property in Definition \ref{good}) it is enough to consider the tame open subsets. The general case follows by an inverse limit argument and by the goodness of $F_{1},F_{2}$. \\
Let $V\in\mathcal{O}$ be a tame open subset of $K$ and let $f_{V}:\N\rightarrow\N\cup\left\{-\infty\right\}$ be the handle index function of $V$. 
We induct on the following statement depending on $j$: The map $F_{1}(V)\rightarrow F_{2}(V)$ is a weak equivalence for all tame open sets $V\in\mathcal{O}$ with $f_{V}(m)\leq 0$ for all $m>j$. \\
The induction starts with the statement for $j=0$, i.e. $f_{V}(m)\leq 0$ for all $m\in\N$. This means that there is an integer $r$ such that $V\in\mathcal{O}r$. If $r\leq k$, then we have a weak equivalence $F_{1}(V)\rightarrow F_{2}(V)$ by assumption. If $r= k+1$, we can find exactly $k+1$ components $A_{0},...,A_{k}$ of $V$. For $T\subset [k]$ we define $V_{T}:= V\setminus \cup_{i\in T} A_{i}$. By assumption, the maps
\begin{align*}
F_{i}(V) \rightarrow \holimsub{\emptyset\neq T\subset[k]} F_{i}(V_{T})
\end{align*}
are weak equivalences for $i=1,2$. We consider the following commutative diagram
\begin{equation*}
\begin{gathered}
\xymatrix{
F_{1}(V) \ar[d] \ar[r] & \text{holim}_{\emptyset\neq T\subset[k]} F_{1}(V_{T}) \ar[d] \\
F_{2}(V) \ar[r] & \text{holim}_{\emptyset\neq T\subset[k]} F_{2}(V_{T})
}
\end{gathered}
\end{equation*}
The map $F_{1}(V_{T})\rightarrow F_{2}(V_{T})$ is a weak equivalence for every $\emptyset\neq T\subset[k]$ and thus we have proven that $F_{1}(V)\rightarrow F_{2}(V)$ is a weak equivalence for all $V\in\mathcal{O}(k+1)$. Likewise, we get weak equivalences $F_{1}(V)\rightarrow F_{2}(V)$ for all $V\in\mathcal{O}r$ and for all integers $r$. \\
Now we assume that the statements $0,1,2,...,j-1$ are proven and we suppose that $f_{V}(j)=q$ for a fixed integer $q>0$ and $f_{V}(m)\leq 0$ for all $m>j$. \\
Since $V$ is tame, there is a codimension zero compact subobject $C\subset K$ with $V=\text{int}(C)$. For every handle $Q_{u}$ of index $q$ which is a subset of a $j$-simplex $S_{u}$, choose a diffeomorphism 
\[
h_{u}: D^{q}\times D^{j-q} \rightarrow Q_{u}\subset C\cap S_{u} \subset C\cap K^{j}
\]
Since $q>0$, there are distinct points $x_{0}^{u},...,x_{k}^{u}$ in the interior of $D^{q}$. We set
\begin{align*}
A_{i}^{u}:= h_{u}(x_{i}^{u}\times D^{j-q})\cap V 
\end{align*}

Define $A_{i}$ to be the union of all collars $\text{col}_{V}(A_{i}^{u})$ of $A_{i}^{u}$ for arbitrary $u$. \\
By definition, $A_{i}$ is a closed subset of $V$ for each $i$. If we set $V_{T}:=V\setminus\cup_{i\in S}A_{i}$ for $\emptyset\neq T\subset[k]$, then $V_{T}$ is a tame open set with $f_{V}(j)<q$ and $f_{V}(m)=0$ for all $m>j$. We can use the induction hypothesis and we deduce that the map $F_{1}(V_{T})\rightarrow F_{2}(V_{T})$ is a weak equivalence for all $\emptyset\neq T\subset[k]$. Consider the commutative square
\begin{equation*}
\begin{gathered}
\xymatrix{
F_{1}(V) \ar[d] \ar[r] & \text{holim}_{\emptyset\neq T\subset[k]} F_{1}(V_{T}) \ar[d] \\
F_{2}(V) \ar[r] & \text{holim}_{\emptyset\neq T\subset[k]} F_{2}(V_{T})
}
\end{gathered}
\end{equation*}
We have shown that the right vertical arrow is a weak equivalence. The horizontal arrows are also weak equivalences since $F_{1}$ and $F_{2}$ are $k$-polynomial. By the commutativity of the diagram, the left vertical arrow is a weak equivalence. By induction on $q$, the statement $j$ is proven. And again by induction (on $j$), the map $F_{1}(V)\rightarrow F_{2}(V)$ is a weak equivalence for all tame open sets $V\in\mathcal{O}$. 
\end{proof}

\section{Taylor tower}

Let $F$ be a good contravariant functor from $\mathcal{O}$ to $(Top)$. In this section we will define the Taylor tower of $F$ by analogy with the Taylor tower in manifold calculus. Most of the ideas of the proof are not new and can be found in \cite[\S 3 and \S 4]{mancal1}. After introducing it, we will show that the new Taylor tower generalizes the Taylor tower in the sense of manifold calculus.

\subsection{Double categories}

We give a brief introduction on double categories, for more details we refer to \cite[ 12.1]{maclane}. A \textit{double category} (or \textit{internal category}) $\mathcal{C}=(\mathcal{C}_{0},\mathcal{C}_{1},i,s,t,\circ )$ consists of two categories $\mathcal{C}_{0}$ and $\mathcal{C}_{1}$ and four functors $i:\mathcal{C}_{0}\rightarrow \mathcal{C}_{1}$ (inclusion functor), $s:\mathcal{C}_{1}\rightarrow \mathcal{C}_{0}$ (source), $t:\mathcal{C}_{1}\rightarrow \mathcal{C}_{0}$ (target) and $\circ : \mathcal{C}_{1} \times_{\mathcal{C}_{0}} \mathcal{C}_{1} \rightarrow \mathcal{C}_{1}$ (composition functor) where $\mathcal{C}_{1} \times_{\mathcal{C}_{0}} \mathcal{C}_{1}$ denotes the pullback of the pullback square
\begin{equation*}
\begin{gathered}
\xymatrix{
\mathcal{C}_{1} \times_{\mathcal{C}_{0}} \mathcal{C}_{1} \ar[d] \ar[r] & \mathcal{C}_{1} \ar[d]^{s} \\
\mathcal{C}_{1} \ar[r]^{t} & \mathcal{C}_{0}
}
\end{gathered}
\end{equation*}
The four functors have to fulfil various relations. \\
If $\mathcal{C}$ is a double category, its \textit{nerve} $\left|\mathcal{C}\right|$ is defined to be a bisimplicial set in the obvious way. \\
Let $\mathcal{C}=(\mathcal{C}_{0},\mathcal{C}_{1},i,s,t,\circ )$ and $\mathcal{C}'=(\mathcal{C}_{0}',\mathcal{C}_{1}',i',s',t',\circ ' )$ be two double categories. A \textit{double functor} (or \textit{internal functor}) $F:\mathcal{C}\rightarrow\mathcal{D}$ is a pair of functors $F_{0}:\mathcal{C}_{0}\rightarrow\mathcal{D}_{0}$ and $F_{1}:\mathcal{C}_{1}\rightarrow\mathcal{D}_{1}$ that fulfil the expected relations.

\begin{ex}
We consider $[p]$ as a category: The objects are the elements of $[p]$. For $p_{1},p_{2}\in[p]$, there is exactly one morphism $p_{1}\rightarrow p_{2}$ if $p_{1}\leq p_{2}$, otherwise the morphism-set is empty. Then we can consider $[p]\times[q]$ as a double category: $C_{0}$ is the category where the objects are the elements of $[p]\times[q]$ and the morphisms are the horizontal arrows, i.e. they do not change the second coordinate. The objects of $C_{1}$ are the vertical morphisms in $[p]\times[q]$ which do not change the first coordinate and the morphisms of $C_{1}$ are commutative squares
\begin{equation*}
\begin{gathered}
\xymatrix{
(p_{1},q_{1}) \ar[d] \ar[r] & (p_{2},q_{1}) \ar[d] \\
(p_{1},q_{2}) \ar[r] & (p_{2},q_{2})
}
\end{gathered}
\end{equation*}
where $p_{1}\leq p_{2}$ and $q_{1}\leq q_{2}$, i.e. the vertical arrows are morphisms in $C_{0}$.
\end{ex}

\begin{ex}\label{doubleDC}
Let $\mathcal{C}$ be an arbitrary category and let $\text{ar}(\mathcal{C})$ be the arrow category of $\mathcal{C}$. More precisely, the objects of the arrow category of $\mathcal{C}$ are the morphisms in $\mathcal{C}$ and a morphism between two objects $f:x\rightarrow y$ and $g:z\rightarrow w$ of $\text{ar}(\mathcal{C})$ is a commutative square in $\mathcal{C}$
\begin{equation*}
\begin{gathered}
\xymatrix{
x \ar[d] \ar[r]^{f} & y \ar[d] \\
z \ar[r]^{g} & w
}
\end{gathered}
\end{equation*}
Now we have a double category $(\mathcal{C},\text{ar}(\mathcal{C}),i,s,t,\circ)$ where $i$ maps an object of $\mathcal{C}$ to its identity morphism and $s,t,\circ$ are the usual source-, target- and composition functor. \\
More generally, given a category $\mathcal{C}$ and subcategory $\mathcal{D}$ containing all objects of $\mathcal{C}$. We define the category $\text{ar}_{\mathcal{D}}(\mathcal{C})$ as follows: The objects are the morphisms in $\mathcal{C}$ and a morphism between two objects $f:x\rightarrow y$ and $g:z\rightarrow w$ of $\text{ar}_{\mathcal{D}}(\mathcal{C})$ is a commutative square
\begin{equation*}
\begin{gathered}
\xymatrix{
x \ar[d] \ar[r]^{f} & y \ar[d] \\
z \ar[r]^{g} & w
}
\end{gathered}
\end{equation*}
where the vertical arrows are morphisms in $\mathcal{D}$. Then we have a double category $(\mathcal{D},\text{ar}_{\mathcal{D}}(\mathcal{C}),i,s,t,\circ)$ where $i$ maps an object of $\mathcal{C}$ to its identity morphism and $s,t,\circ$ are the usual source-, target- and composition functor. 
We denote this double category by $\mathcal{D}\mathcal{C}$.
\end{ex}

The next two lemmas are proven in \cite[Lemma 3.3 + 3.4]{mancal1}.
\begin{lem}
The inclusions of nerves $\left|\mathcal{C}\right|\rightarrow \left|\mathcal{D}\mathcal{C}\right|$ is a weak equivalence.
\end{lem}
\begin{rmk}
We will need the totalization of a bicosimplicial space. Firstly, we would like to remind the reader that the totalization of a cosimplicial space $C^{\bullet}$ is just the space of natural transformations from the cosimplicial space $\Delta^{\bullet}$ to $C^{\bullet}$. \\
Let $B^{\bullet,\bullet}$ be a bicosimplicial space. Then the totalization of $B^{\bullet,\bullet}$ is the space of natural transformations from the bicosimplicial space $\Delta^{\bullet}\times \Delta^{\bullet}$ to $B^{\bullet,\bullet}$.
\end{rmk}
Let $F$ be a double functor from a double category $\mathcal{A}$ to the double category $(Top)(Top)$ (compare Example \ref{doubleDC}). Then we define the homotopy limit $\holim_{\mathcal{A}} F$ as the totalization of the bicosimplicial space
\[
(p,q)\mapsto \prod_{H:[p]\times[q]\rightarrow\mathcal{A}} F(H(p,q))
\]
where the product ranges over all double functors $H$ from $[p]\times[q]$ to $\mathcal{A}$. \\

Now let $F$ be a functor from the category $\mathcal{C}$ to $(Top)$. Then $F$ can also be considered as a double functor from the double category $\mathcal{D}\mathcal{C}$ to $(Top)(Top)$.
\begin{lem}\label{prdb}
If $F$ takes all morphisms in $\mathcal{D}$ to weak equivalences, the projection map
\[
\holim_{\mathcal{D}\mathcal{C}} F \rightarrow \holim_{\mathcal{C}} F
\]
is a weak equivalence.
\end{lem}
Let $\mathcal{C}$ be a small category and $\mathcal{D}$ be a subcategory containing all objects of $\mathcal{C}$. Then for every $p\geq 0$, we introduce a new category $\mathcal{D}\mathcal{C}_{p}$: the objects are functors $G : \left[ p \right] \rightarrow \mathcal{C}$ and the morphisms are double functors $\left[ 1 \right] \times \left[ p \right] \rightarrow \mathcal{D}\mathcal{C}$. 

\begin{lem}\label{dftcs}
Let $F$ be a double functor from $\mathcal{D}\mathcal{C}$ to $(Top)(Top)$. There is an isomorphism between $\holim_{\mathcal{D}\mathcal{C}} F$ and the totalization of the cosimplicial space
\begin{align*}
p\mapsto \holimsub{G:\left[ p \right] \rightarrow \mathcal{C}} F ( G(p) )
\end{align*}
where the homotopy limit ranges over all $G\in \mathcal{D}\mathcal{C}_{p}$.
\end{lem}
\begin{proof}
We just need to compare the definitions:
\begin{align*}
\holim_{\mathcal{D}\mathcal{C}} F 
&= \Tot \left( (p,q)\mapsto \prod_{H:[p]\times[q]\rightarrow\mathcal{A}} F(H(p,q)) \right) \\
&\cong \Tot \left( p \mapsto \Tot \left( q \mapsto \prod_{H:[p]\times[q]\rightarrow\mathcal{A}} F(H(p,q))\right) \right) \\
&= \Tot \left( p \mapsto \holimsub{G:\left[ p \right] \rightarrow \mathcal{C}} F ( G(p) ) \right) 
\end{align*}
Note that in the first line we consider the totalization of a bicosimplicial space, while the other totalizations are built out of cosimplicial spaces.
\end{proof}

\subsection{The Homotopy Kan extension is polynomial}

In this section we will prove that the homotopy Kan extension of a good functor along the inclusion $\mathcal{O}k \hookrightarrow \mathcal{O}$ is $k$-polynomial. Most parts of the proof follow similar lines as its analogue in Goodwillie-Weiss calculus. For the sake of completeness, we also provide these parts.

\begin{defi}
Let $X$ be a topological space and $r$ be a positive integer. We define the space $F(X,r)$ of ordered configurations of $X$ by
\[
F(X,r) := \left\{ (x_{1},...,x_{r})\in X^{r} \mid x_{i}\neq x_{j} \hspace{0,2cm}\text{for all}\hspace{0,2cm} i\neq j \right\}
\]
The symmetric group $\Sigma_{r}$ acts freely on $F(X,r)$. Let 
\[
B(X,r) := F(X,r) / \Sigma_{r}
\]
be the space of unordered configurations.
\end{defi}

Let $\epsilon$ be an open cover of $K$. 

\begin{defi}
Let $V\in \mathcal{O}k$ be given. Then $V$ is \textit{$\epsilon$-small} if for each connected component $V_{0}$ of $V$, there is an $U\in\epsilon$ such that $V_{0}\subset U$.
\end{defi}

\textbf{Notations:} Let $\mathcal{I}k$ be the subcategory of $\mathcal{O}k$ consisting of the same objects and all morphisms that are stratified isotopy equivalences. \\ 
Let $\epsilon\mathcal{O}k$ be the full subcategory of $\mathcal{O}k$ consisting of the $\epsilon$-small objects. Similarly, we define $\epsilon\mathcal{I}k$ to be the full subcategory of $\mathcal{I}k$ consisting of the $\epsilon$-small objects. 
For $V\in\mathcal{O}(K)$, we introduce $\epsilon\mathcal{O}k(V)$, respectively $\epsilon\mathcal{I}k(V)$, to be the full subcategory of $\epsilon\mathcal{O}k$, respectively $\epsilon\mathcal{I}k$, with all objects which are subsets of $V$. \\

The next lemma gives us the homotopy type of $\left|\epsilon\mathcal{I}k(V) \right|$.

\begin{lem}\label{configurations}
For all $V\in\mathcal{O}(K)$, the following spaces are (weakly) equivalent:
\begin{align*}
\left|\epsilon\mathcal{I}k (V) \right|  \simeq  \coprod_{(S_{1},k_{1}),...,(S_{l},k_{l})} B(\op(S_{1})\cap V,k_{1}) \times ... \times B(\op(S_{l})\cap V,k_{l})
\end{align*}
The disjoint union ranges over all pairs $(S_{i},k_{i})$, $1\leq i\leq l$, where $\op(S_{i})$, $1\leq i\leq l$, are disjoint \textbf{open} simplices of $K$ and $\sum_{i=1}^{l} k_{i} \leq k$. \\
In particular, the functor $V\mapsto \left| \epsilon\mathcal{I}k (V) \right|$ takes stratified isotopy equivalences to weak equivalences. 
\end{lem}
Note: As a set the above disjoint union is equal to the disjoint union of all configuration spaces $B(V,j)$ with $0\leq j\leq k$. The complicated topology comes from morphisms in $\epsilon\mathcal{I}k (V) $, i.e. from the definition of stratified isotopy equivalences.
\begin{proof}
For $0\leq j\leq k$, let $\epsilon\mathcal{I}^{(j)}(V)$ be the full subcategory of $\epsilon\mathcal{I}k(V)$ where the objects are all open subsets in $\epsilon\mathcal{I}k(V)$ which have exactly $j$ components. Then $\epsilon\mathcal{I}k(V)$ is a coproduct $\coprod_{0\leq j\leq k} \hspace{1mm}\epsilon\mathcal{I}^{(j)}(V)$. We have to determine the homotopy type of $\left|\epsilon\mathcal{I}^{(j)}\right|$.
For $j=0$, this is obvious, thus let $j=1$. In this case, there is a one-one correspondence between the components of $\left|\epsilon\mathcal{I}^{(1)}\right|$ and the open simplices of $K$ (see Remark \ref{compI1}). Claim: For $V\in\mathcal{O}$,
\begin{align*}
\left|\epsilon\mathcal{I}^{(1)}(V)\right| \simeq \coprod_{S} \left|\epsilon\mathcal{I}^{(1)}(\op(S)\cap V)\right| 
\end{align*}
where the disjoint union ranges over all simplices $S$ with $\op(S)\cap V\neq \emptyset$. Here $\op(S)\cap V$ can be considered as a manifold (without boundary). Obviously, $\left|\epsilon\mathcal{I}^{(1)}(V)\right|$ has one component for each simplex $S$ of $K$with $\op(S)\cap V\neq \emptyset$ - namely the classifying space of all $U\in\epsilon\mathcal{I}^{(1)}(V)$ with $U\cap \op(S) \neq \emptyset$ and $U\cap\partial S =\emptyset$. Therefore, we can concentrate on one simplex $S$ with this property. If $\op(S)$ is open in $K$, it is also obvious that the corresponding component of $\left|\epsilon\mathcal{I}^{(1)}(V)\right|$ is (weakly) equivalent to $\left|\epsilon\mathcal{I}^{(1)}(\op(S)\cap V)\right|$ (it is even equal). If $S$ is a subsimplex of another simplex, each element $U$ of $\epsilon\mathcal{I}^{(1)}(V)$ with $U\cap \op(S) \neq \emptyset$ and $U\cap\partial S =\emptyset$ is a collar of $U\cap \op(S)$. But this is (weakly) equivalent to $\left|\epsilon\mathcal{I}^{(1)}(\op(S))\right|$: There is a homotopy terminal functor
\[
\left\{ U\in \epsilon\mathcal{I}^{(1)}(V) \mid U\cap \op(S) \neq \emptyset, U\cap\partial S =\emptyset \right\}   \hspace{0,2cm}\rightarrow\hspace{0,2cm}   \epsilon\mathcal{I}^{(1)}(\op(S))
\]
which is given by $U \mapsto U\cap \op(S)$ (this is not obvious). Therefore, the corresponding map of classifying spaces is a weak equivalence. Now we can use the analogue in (smooth) manifold calculus \cite[3.5]{mancal1} and we get 
\begin{align*}
\left|\epsilon\mathcal{I}^{(1)}(V)\right| \simeq \coprod_{S} (\op(S)\cap V)
\end{align*}
The case $j>1$ follows similar lines, but is even more complicated. Therefore, we will provide another proof. \\

There is another approach to verify the weak equivalence
\[
\Phi^{(1)} := \left|\left\{ U\in \epsilon\mathcal{I}^{(1)}(V) \mid U\cap \op(S) \neq \emptyset, U\cap\partial S =\emptyset \right\}\right| \hspace{0,2cm} \simeq  \hspace{0,2cm} \op(S)\cap V
\]
which is similar to the proof of \cite[Lemma 3.5]{mancal1} and does not use that the above functor is homotopy terminal. Let 
\[
E \subset \Phi^{(1)} \times (\op(S)\cap V)
\]
be the space of all pairs $(x,y)$ such that the open cell containing $x$ corresponds to the simplex
\[
U_{0}\rightarrow ... \rightarrow U_{r}
\]
and $y\in \op(S)\cap U_{r}$. We consider the projection maps 
\[
\Phi^{(1)} \leftarrow E \rightarrow (\op(S)\cap V)
\]
We have to verify that these maps are weak equivalences. We skip the verification because it is analogous to the proof of \cite[Lemma 3.5]{mancal1}. \\
For $j>1$, there is a one-one correspondence between the components of $\left|\epsilon\mathcal{I}^{(j)}\right|$ and the set $\Omega^{(j)}$ of all collections of pairs $(S_{i},k_{i})$, $1\leq i\leq l$, where $\op(S_{i})$, $1\leq i\leq l$, are disjoint open simplices of $K$ and $\sum_{i=1}^{l} k_{i} = j$. Next, we have to prove that there is an equivalence
\begin{align*}
\left|\epsilon\mathcal{I}^{(j)} (V) \right|  \simeq   \coprod_{\Omega^{(j)}} \left|\epsilon\mathcal{I}^{(k_{1})}(\op(S_{1})\cap V)\right| \times ... \times \left|\epsilon\mathcal{I}^{(k_{l})}(\op(S_{l})\cap V)\right|
\end{align*}
This can be shown in the following way: Let $(S_{i},k_{i})$, $1\leq i\leq l$, be an element of $\Omega^{(j)}$. Then we define $\Phi^{(j)}$ to be the following component of $\left|\epsilon\mathcal{I}^{(j)} (V) \right|$: it is the classifying space of all $U \in \epsilon\mathcal{I}^{(j)}(V)$ such that for every $1\leq i\leq l$, $U$ has exactly $k_{i}$ components which have nonempty intersection with $\op(S_{i})$ and empty intersection with $\partial S_{i}$. Then we consider the space 
\[
E \subset \Phi^{(j)} \times \left(B(\op(S_{1})\cap V,k_{1}) \times ... \times B(\op(S_{l})\cap V,k_{l})\right)
\]
of all pairs $(x,T)$ such that the open cell containing $x$ corresponds to the simplex
\[
U_{0}\rightarrow ... \rightarrow U_{r}
\]
where each component of $U_{r}$ contains exactly one point of $T$. Analogously to the case $j=1$, we can prove that the projection maps 
\[
\Phi^{(j)} \leftarrow E \rightarrow \left(B(\op(S_{1})\cap V,k_{1}) \times ... \times B(\op(S_{l})\cap V,k_{l})\right)
\]
are weak equivalences.
\end{proof}

Let $\mathcal{C}$ be the category $\epsilon\mathcal{O}k$ and $\mathcal{D}$ be the subcategory $\epsilon\mathcal{I}k$. Now we consider the double category $\epsilon\mathcal{I}k\mathcal{O}k:= \mathcal{D}\mathcal{C}$ (compare Example \ref{doubleDC}).  \\
\textbf{Notation:} The category $\epsilon\mathcal{I}k\mathcal{O}k_{p}(V)$ is a full subcategory of $\epsilon\mathcal{I}k\mathcal{O}k_{p}$ with all objects 
\[
(U_{0}\subset U_{1} \subset ... \subset U_{p})\in\epsilon\mathcal{I}k\mathcal{O}k_{p}
\]
such that $U_{i}\subset V$ for all $i\in[p]$. \\
There is a functor from $\epsilon\mathcal{I}k\mathcal{O}k_{p}(V)$ to $\epsilon\mathcal{I}k(V)$ given by $G\mapsto G(p)$ where $G:[p]\rightarrow \mathcal{O}k(V)$ is an element of $\epsilon\mathcal{I}k\mathcal{O}k_{p}(V)$. The following lemma gives an idea of the homotopy type of $\epsilon\mathcal{I}k\mathcal{O}k_{p}(V)$.

\begin{lem}\label{lemisotohom}
The following two conditions are fulfilled:
\begin{enumerate}
 \item Given $U,V\in\epsilon\mathcal{I}k(K)$ with $U\subset V$, there is a homotopy equivalence between $\left| \epsilon\mathcal{I}k\mathcal{O}k_{p-1}(U) \right|$ and the homotopy fiber over the point (which is identified with) $U$ of the map
\begin{align*}
\left| \epsilon\mathcal{I}k\mathcal{O}k_{p} (V) \right| \rightarrow \left| \epsilon\mathcal{I}k(V) \right|
\end{align*}
induced by $G\mapsto G(p)$. 
 \item The functor $V\mapsto \left| \epsilon\mathcal{I}k\mathcal{O}k_{p}(V) \right|$ takes stratified isotopy equivalences to weak equivalences.
\end{enumerate}
\end{lem}
\begin{proof}
We prove these two statements parallelly by induction on $p$. For $p=0$, we can use Lemma \ref{configurations}. \\ 
By induction, we assume that the functor $V\mapsto \left| \epsilon\mathcal{I}k\mathcal{O}k_{p-1}(V) \right|$ takes stratified isotopy equivalences to weak equivalences. Using Thomason`s homotopy colimit theorem \cite{thomass}, the map under investigation which is induced by $G\mapsto G(p)$ corresponds to the canonical map
\begin{align*}
\hocolimsub{U\in\epsilon\mathcal{I}k(V)} \left| \epsilon\mathcal{I}k\mathcal{O}k_{p-1}(U) \right| \rightarrow \left| \epsilon\mathcal{I}k(V) \right| 
\end{align*}
By Prop. \ref{quasifibration}, this map is a quasifibration. Therefore, the homotopy fiber coincides (up to homotopy) with the fiber. The fiber of this map over $U$ is evidently $\left| \epsilon\mathcal{I}k\mathcal{O}k_{p-1}(U) \right|$. Using the resulting (homotopy) fiber sequence, it follows that the functor $V\mapsto \left| \epsilon\mathcal{I}k\mathcal{O}k_{p}(V) \right|$ takes stratified isotopy equivalences to weak equivalences, too. 
\end{proof}

\textbf{Notation:} Let $F : \epsilon\mathcal{O}k \rightarrow (Top)$ be a contravariant functor which takes all stratified isotopy equivalences to weak equivalences. Then we define the contravariant functor $\epsilon F^{!} : \mathcal{O} \rightarrow (Top)$ by
\begin{align*}
\epsilon F^{!}(V) := \holimsub{U\in \epsilon\mathcal{O}k(V)} F(U)
\end{align*}
By definition, $\epsilon F^{!}$ is the homotopy right Kan extension along the inclusion functor $\epsilon\mathcal{O}k \rightarrow \mathcal{O}$.

\begin{lem}\label{goodfunctor}
The functor $\epsilon F^{!}$ is good.
\end{lem}
\begin{proof}
By Lemma \ref{prdb}, the projection map
\begin{align*}
\holimsub{U\in\epsilon\mathcal{I}k\mathcal{O}k(V)} F(U) \rightarrow \holimsub{U\in\epsilon\mathcal{O}k(V)} F(U)
\end{align*}
is a weak equivalence.
By Lemma \ref{dftcs}, we have an isomorphism between $\holim_{\epsilon\mathcal{I}k\mathcal{O}k(V)} F$ and the totalization of the cosimplicial space
\begin{align*}
p\mapsto \holimsub{(G:\left[ p \right] \rightarrow \epsilon\mathcal{O}k) \in\epsilon\mathcal{I}k\mathcal{O}k_{p}(V)} F ( G(p) )
\end{align*}
Note that the functor from $\epsilon\mathcal{I}k\mathcal{O}k_{p}(V)$ to $(Top)$ given by $G\mapsto F(G(p))$ takes all morphisms to weak equivalences. Therefore, the canonical map 
\begin{align*}
\hocolimsub{(G:\left[ p \right] \rightarrow \epsilon\mathcal{O}k) \in\epsilon\mathcal{I}k\mathcal{O}k_{p}(V)} F(G(p)) \rightarrow \left|\epsilon\mathcal{I}k\mathcal{O}k_{p}(V)\right|
\end{align*}
is a quasifibration (Proposition \ref{quasifibration}). 
Using Proposition \ref{holimassfib}, the section space of the associated fibration is weakly equivalent to 
\[
\holimsub{(G:\left[ p \right] \rightarrow \epsilon\mathcal{O}k) \in\epsilon\mathcal{I}k\mathcal{O}k_{p}(V)} F ( G(p) )
\]
Now let $V_{0}\rightarrow V_{1}$ be a morphism in $\epsilon\mathcal{I}k$. Using Lemma \ref{lemisotohom}, the inclusion of categories $\epsilon\mathcal{I}k\mathcal{O}k_{p}(V_{0}) \rightarrow \epsilon\mathcal{I}k\mathcal{O}k_{p}(V_{1})$ induces a weak equivalence of classifying spaces. Therefore, the map
\begin{align*}
\hocolimsub{(G:\left[ p \right] \rightarrow \epsilon\mathcal{O}k) \in\epsilon\mathcal{I}k\mathcal{O}k_{p}(V_{0})} F ( G(p) )
\rightarrow
\hocolimsub{(G:\left[ p \right] \rightarrow \epsilon\mathcal{O}k) \in\epsilon\mathcal{I}k\mathcal{O}k_{p}(V_{1})} F ( G(p) )
\end{align*}
is also a weak equivalence (use Proposition \ref{homequsmcat}). We have shown that 
\[
V\mapsto \holimsub{(G:\left[ p \right] \rightarrow \epsilon\mathcal{O}k) \in\epsilon\mathcal{I}k\mathcal{O}k_{p}(V)} F ( G(p) )
\] 
is a good functor for all $p$. Therefore, $\epsilon F^{!}$ is a good functor.
\end{proof}

\textbf{Notation:} If $\epsilon = \left\{ K \right\}$, then $\epsilon\mathcal{O}k(V)=\mathcal{O}k(V)$ for all $V\in\mathcal{O}(K)$. We define
\begin{align*}
F^{!}(V) := \holimsub{U\in\mathcal{O}k(V)} F(U)
\end{align*}

\begin{thm}\label{opencover}
The induced map $F^{!}(V) \rightarrow \epsilon F^{!}(V)$ is a weak equivalence.
\end{thm} 
\begin{proof}
Using Lemma \ref{prdb} and Lemma \ref{dftcs}, it suffices to show that there are weak equivalences 
\[
\holimsub{U\in\mathcal{I}k\mathcal{O}k_{p}(V)} F(U) \rightarrow \holimsub{U\in\epsilon\mathcal{I}k\mathcal{O}k_{p}(V)} F(U)
\]
for all $p$. We consider the following composition of maps: 
\begin{center}
$\holimsub{U\in\mathcal{I}k\mathcal{O}k_{p}(V)} F(U)$ \\
\vspace{0,4cm}
$\simeq$ \\
\vspace{0,4cm}
$\left\{ s : \left|\mathcal{I}k\mathcal{O}k_{p}(V)\right| \rightarrow\hocolimsub{G\in\mathcal{I}k\mathcal{O}k_{p}(V)} F(G(p)) \mid pr\circ s = id_{ \left|\mathcal{I}k\mathcal{O}k_{p}(V)\right|} \right\}$ \\
$\downarrow$ \\
$\left\{ s : \left|\epsilon\mathcal{I}k\mathcal{O}k_{p}(V)\right| \rightarrow\hocolimsub{G\in\epsilon\mathcal{I}k\mathcal{O}k_{p}(V)} F(G(p)) \mid pr\circ s = id_{ \left|\epsilon\mathcal{I}k\mathcal{O}k_{p}(V)\right|} \right\}$ \\
\vspace{0,4cm}
$\simeq$ \\
\vspace{0,4cm}
$\holimsub{U\in\epsilon\mathcal{I}k\mathcal{O}k_{p}(V)} F(U)$
\end{center}
The (weak) equivalences are the equivalences given by Theorem \ref{holimassfib}. The map between the section spaces is given by restriction (note that $\left|\epsilon\mathcal{I}k\mathcal{O}k_{p}(V)\right|$ is a subset of $\left|\mathcal{I}k\mathcal{O}k_{p}(V)\right|$). Therefore, the composition is the canonical map (up to homotopy). In order to verify that the second map is a weak equivalence, we use Theorem \ref{homequsmcat} (by Lemma \ref{configurations} and Lemma \ref{lemisotohom}, the inclusion of categories
\[
\epsilon\mathcal{I}k\mathcal{O}k_{p}(V) \rightarrow \mathcal{I}k\mathcal{O}k_{p}(V)
\]
induces a weak equivalence of classifying spaces).
\end{proof}

\begin{cor}\label{kpolyn}
The functor $F^{!}:\mathcal{O}\rightarrow (Top)$ is polynomial of degree $\leq k$.
\end{cor}
\begin{proof}
We have to show that the condition in Definition \ref{polynomial} is satisfied. Let $V\in\mathcal{O}$ be an open set and $A_{0},A_{1},...,A_{k}$ be pairwise disjoint closed subsets of $V$. Without loss of generality, we assume $V=K$ (the general proof follows similar lines). \\
Now we define $K_{T}:=\cap_{i\in T} (K\setminus A_{i})$ for $T\subset [k]=\left\{0,1,...,k\right\}$ and the open cover $\epsilon:= \left\{ K_{T} \mid k=\left|T\right|\right\}$ of $K$. For each $U\in \epsilon\mathcal{O}k$, there is an $i\in[k]$ such that $U\cap A_{i} = \emptyset$ (pigeonhole principle: each component of $U$ meets at most one of the $A_{j}$, but $U$ has at most $k$ components). It follows
\begin{align*}
\epsilon\mathcal{O}k(K)=\cup_{i\in[k]}\hspace{0,1cm}\epsilon\mathcal{O}k(K_{\left\{i\right\}})
\end{align*}
Now we can use \cite[Lemma 4.2]{mancal1} and follow that the canonical map
\begin{align*}
\epsilon F^{!}(K) = \holimsub{\epsilon \mathcal{O}k} F \rightarrow \holimsub{T\neq \emptyset} \holimsub{\epsilon\mathcal{O}k(K_{T})} F = \holimsub{T\neq \emptyset} F^{!}(K_{T})
\end{align*}
is a weak equivalence. We have shown that the $k$-cube 
\begin{align*}
S\mapsto \epsilon F^{!} (K_{T})
\end{align*}
is homotopy cartesian. By Theorem \ref{opencover}, the functor $F^{!}$ is polynomial of degree $\leq k$.
\end{proof}

\subsection{The tower}

Let $F$ be a contravariant good functor from $\mathcal{O}$ to $(Top)$. For every $k\geq 0$, we define the functor $T_{k}F$ from $\mathcal{O}$ to $(Top)$ by
\[
T_{k}F(V) := \holimsub{U\in\mathcal{O}k(V)} F(U)
\]
which is called the \textit{$k$-th Taylor approximation} of $F$. By definition, there is a canonical transformation $\eta_{k}:F\rightarrow T_{k}F$. The following proposition follows from Theorem \ref{F1F2} and Corollary \ref{kpolyn}.
\begin{prop}
If $F$ is $k$-polynomial, the canonical map
\[
\eta_{k}(V) : F(V)\rightarrow T_{k}F(V) 
\]
is a weak equivalence for every open set $V\in\mathcal{O}$.
\end{prop}

By analogy with the manifold case \cite{mancal1} we can define a Taylor tower. More precisely, there are forgetful transformations
\[
r_{k}: T_{k}F\rightarrow T_{k-1}F
\]
for all $k$ which make up a tower. The functor $F$ maps into this tower in a natural way: 
\[
r_{k}\eta_{k}=\eta_{k-1}:F\rightarrow T_{k-1}F
\]
Therefore, the transformations $\eta_{k}$ induce a transformation
\[
\eta_{\infty} : F\rightarrow \text{holim}_{k}T_{k}F
\]
In the next section we ask about convergence, i.e. we ask whether the map $\eta_{\infty} : F(V)\rightarrow \text{holim}_{k}T_{k}F(V)$ is a weak equivalence for some $V\in\mathcal{O}$. \\
Now we want to compare this new Taylor tower with the old one constructed in \cite{mancal1}. Therefore, let $M$ be a smooth manifold of dimension $m$, let $K$ be a triangulation of $M$ and let $F:\mathcal{O}(M)\rightarrow (Top)$ be a good (contravariant) functor in the sense of \cite{mancal1}.  \\ 
Now let $\mathcal{O}k(M)$ be the set of special open subsets of $M$ with no more than $k$ components. More precisely, $\mathcal{O}k(M)$ is a full subcategory of $\mathcal{O}(M)$ where the objects are all open subsets $U$ of $M$ such that $U$ is diffeomorphic to a disjoint union of $r$ copies of $\R^{m}$ for a positive integer $r\leq k$. By definition, we have an inclusion of categories $\mathcal{O}k(K)\rightarrow\mathcal{O}k(M)$ which induces a canonical projection of homotopy limits.

\begin{thm}\label{vergltt}
For all $V\in\mathcal{O}(K)=\mathcal{O}(M)$, the canonical map
\[
\holimsub{U\in\mathcal{O}k(M),U\subset V} F(U) \rightarrow \holimsub{U\in\mathcal{O}k(K),U\subset V} F(U)
\]
is a weak equivalence. Therefore, the Taylor tower in the sense of manifold calculus \cite{mancal1} coincides with the Taylor tower in this new setting.
\end{thm}
\begin{proof}
For simplicity we assume that $V=M=K$. We have to distinguish between the special open sets in the two calculus versions. As indicated, $\mathcal{O}k(K)$ is the set of special open subsets in this new setting (which was denoted by $\mathcal{O}k$ up to now). The category $\mathcal{I}k(K)$ is the subcategory with the same objects and stratified isotopy equivalences as morphisms. The category $\mathcal{I}k(M)$ is the subcategory of $\mathcal{O}k(M)$ with the same objects and isotopy equivalences in the sense of \cite[Definition 1.1]{mancal1} as morphisms. \\
Let $\mathcal{U}k$ be the full subcategory of $\mathcal{I}k(M)$ where the objects are all open sets $U\in\mathcal{I}k(K)\subset\mathcal{I}k(M)$. We get inclusions 
\[
\mathcal{I}k(K)\rightarrow\mathcal{U}k\rightarrow\mathcal{I}k(M)
\]
of categories. By Lemma 3.5 and Lemma 3.6, we have weak equivalences
\[
\Tot (p\mapsto \holimsub{(\mathcal{U}k)\mathcal{O}k(K)_{p}} F) \hspace{0,2cm}\cong \hspace{0,2cm}\holimsub{(\mathcal{U}k)\mathcal{O}k(K)} F \hspace{0,2cm}\rightarrow \hspace{0,2cm}\holimsub{\mathcal{O}k(K)} F
\]
Similarly, we get weak equivalences
\[
\Tot (p\mapsto \holimsub{\mathcal{I}k(M)\mathcal{O}k(M)_{p}} F)\hspace{0,2cm} \cong\hspace{0,2cm} \holimsub{\mathcal{I}k(M)\mathcal{O}k(M)} F \hspace{0,2cm}\rightarrow \hspace{0,2cm}\holimsub{\mathcal{O}k(M)} F
\]
By \cite[Lemma 3.5]{mancal1}, we know the homotopy type of $\left|\mathcal{I}k(M)\right|$. The same proof gives us the homotopy type of $\left|\mathcal{U}k\right|$: The inclusion of classifying spaces 
$
\left|\mathcal{U}k\right|\rightarrow\left|\mathcal{I}k(M)\right|
$
is a weak equivalence.
Now we can use Lemma 3.10 and we conclude
\[
\left|(\mathcal{U}k)\mathcal{O}k(K)_{p}\right|\rightarrow\left|\mathcal{I}k(M)\mathcal{O}k(M)_{p}\right|
\]
is a weak equivalence for every $p$.
Note that $F$ maps all morphisms of $\mathcal{I}k(M)$ and $\mathcal{U}k$ to weak equivalences. By Proposition \ref{homequsmcat}, the canonical map 
\[
\hocolimsub{(\mathcal{U}k)\mathcal{O}k(K)_{p}} F \rightarrow \hocolimsub{\mathcal{I}k(M)\mathcal{O}k(M)_{p}} F
\]
of homotopy colimits is also a weak equivalence, too. Then the canonical map of homotopy limits is a weak equivalence (use Proposition 6.2), too.
Using the homotopy invariance of the totalization the canonical map
\[
\holimsub{\mathcal{O}k(M)} F \rightarrow \holimsub{\mathcal{O}k(K)} F
\]
is a weak equivalence.
\end{proof}

\section{Convergence}

We will investigate the transformations $F\rightarrow T_{k}F$ for a good functor $F$. We need to introduce analytic functors and the relative handle index.

\subsection{Relative handle index in a simplicial complex}

In order to define the relative handle index function, we will need the following definition.

\begin{defi}
Let $P$ be a codimension zero subobject of $K$. A subset 
\[
A \subset K\setminus \text{int}(P)
\] 
is called a \textit{codimension zero subobject of} $K\setminus \text{int}(P)$ if there is any map $f: K \rightarrow \R$ such that 
\begin{enumerate}
    \item $f|_{S\setminus \text{int}(P)} : S\setminus \text{int}(P) \rightarrow \R$ is smooth for all simplices $S$ of $K$
		\item $A:= f^{-1} \left(\left[0,\infty \right)\right)$
		\item for all simplices $S$ of $K$: $0$ is a regular value for $f|_{\op(S)\setminus\text{int}(P)}$
\end{enumerate}
Then for every simplex $S$ of $K$, $A \cap S$ is a manifold triad (in a non-smooth sense) with $\partial_{0} (A\cap S) = (\partial S \cap A) \cup (\partial (P\cap S) \cap A)$.
\end{defi}

Let $P$ be compact codimension zero subobjects of $K$, let $A$ be a compact codimsion zero subobject of $K\setminus \text{int}(P)$ and let $S_{u}$ be a $j$-simplex in $K$. We set $P_{u}:=P\cap S_{u}$ and $A_{u}:= A\cap S_{u}$ and let $I$ be the finite set of all $u$ with $A_{u}\neq\emptyset$. Then $P_{u}$ and $A_{u}$ are manifolds with boundary. We want to define a handle index function of $A$ which is relative to $P$. Therefore, we consider $A_{u}$ as a manifold triad with 
\begin{center}
$\partial_{0} A_{u} := (\partial S_{u}\cap A_{u}) \cup (\partial P_{u} \cap A_{u})$
\end{center}
and $\partial_{1} A_{u}$ is the closure of $\partial A_{u} \cap \text{int}(S_{u}\setminus P_{u})$ in $A_{u}$. \\

Now we choose a handle decomposition for all $A_{u}$ with $u\in I$. Let $q_{u}$ be the handle index of $A_{u}$ relative to $\partial_{0} A_{u}$. Then we define the \textit{relative handle index function} $f:\N\rightarrow \N\cup\left\{ - \infty \right\}$ (relative to $P$) by 
\begin{align*}
f_{A}(j) := \text{max}_{u\in I(j)} \hspace{0,1cm} q_{u} 
\end{align*}
where $I(j)\subset I$ is the subset of all $u\in I$ such that $S_{u}$ is a $j$-simplex. Furthermore, we call 
\begin{align*}
q_{A}:= \text{max}_{j\in\N} \hspace{0,1cm} f_{A}(j) 
\end{align*}
the \textit{relative handle index} of $A$ (relative to $P$). \\

The reader might find it confusing that we work with the relative handle index function (relative to $P$) and the handle index function in parallel. Note that we defined the relative handle index function $f_{A}$ of a codimension zero subobject $A$ (which is closed by definition). In particular, the boundary $\partial A$ - or more precisely the boundary set $\partial_{0} A=A\cap P$ - is important if we consider the relative handle index. On the other hand, the handle index function $f_{V}$ was defined for a tame open subset $V$ and it depends just on $V$. See Example \ref{relhie}.

\begin{ex}\label{relhie}
Let $K$ be a $1$-dimensional simplicial complex with four $0$-simplices $S_{0},S_{1},S_{2},S_{3}$ and four $1$-simplices $I_{1},I_{2},I_{3},I_{4}$ which are defined by $I_{k}:=\left\{ S_{k-1},S_{k} \right\}$ for $k\in\left\{1,2,3\right\}$ and $I_{4}:= \left\{ S_{3},S_{0} \right\}$. Then we can identify $K$ with the circle $S^{1}=\left\{e^{it}\in\mathbb{C}\mid t\in \left[ 0, 2\pi \right) \right\}$ using the identifications $S_{l} = e^{ \frac{1}{2}it}$ for $l\in\left\{0,1,2,3\right\}$ and $I_{k}=\left\{e^{it}\in\mathbb{C}\mid t\in \left[ \frac{k-1}{2}\pi, \frac{k}{2}\pi \right] \right\}$ for $k\in\left\{1,2,3,4\right\}$. 
Let $P$ be the compact set
\[
P := \left\{e^{it}\in\mathbb{C}\mid t\in \left[ \frac{\pi}{4}, \frac{5 \pi}{4} \right] \right\}
\]
Let $A$ be a codimension zero subobject of $K\setminus \text{int}(P)$ and $f_{A}$ be the relative handle index function. By definition, we have $f_{A}(j)=-\infty$ for all $j\geq 2$. \\
Now we set $A:= \left\{e^{it}\in\mathbb{C}\mid t\in \left[ \frac{7 \pi}{4}, \frac{15 \pi}{8} \right] \right\}$ and determine the relative handle index function of $A$ in this case. It is given by $f_{A}(0)= -\infty$ and $f_{A}(1)= 0$ because $A$ has empty intersection with $P$ and is the closure of a special open set contained in the interior of the $1$-simplex $I_{4}$. \\
Let us consider a more interesting example. We define 
\[
B:= \left\{e^{it}\in\mathbb{C}\mid t\in \left[ 0, \frac{\pi}{4} \right] \cup \left[ \frac{7 \pi}{4} , 2\pi\right] \right\}
\]
The relative handle index function is given by  $f_{B}(0)= 0$ and $f_{B}(1)=1$. Note that the nonzero intersection of $B$ and $P$ leads to $f_{B}(1)=1$.
\end{ex}

\subsection{Analytic functors}

Let $F: \mathcal{O} \rightarrow (Top)$ be a good functor. In the previous subsection we defined the relative handle index for compact codimension zero subobjects of $K$. Now we can define analyticity for $F$. \\
Let $P$ be a compact codimension zero subobject of $K$ and let $\rho$ be a fixed integer. Suppose $A_{0},A_{1},...,A_{r}$ are pairwise disjoint compact codimension zero subobjects of $K\setminus\text{int}(P)$ with relative handle index $q_{A_{i}}\leq \rho$ (relative to $P$). For $T\subset \left\{0,1,...,r\right\}$, we set $A_{T}:= \cup_{i\in T} A_{i}$ and assume $r\geq 1$.
\begin{defi}\label{nalytic}
The functor $F$ is called \textit{$\rho$-analytic} \textit{with excess $c$} if the cube 
\begin{align*}
T\mapsto F(\text{int} \left(P \cup  A_{T})\right) \hspace{0,1cm},\hspace{0,3cm} T\subset \left\{ 0,1,...,r \right\}
\end{align*}
is $c + \sum_{i=0}^{r} (\rho - q_{A_{i}})$-cartesian for some integer $c$.
\end{defi}

\begin{defi} 
The homotopy dimension $\hodim(V)$ of $V\in\mathcal{O}$ is the smallest integer $q$ with the following property:
there is a sequence $\left\{ V_{i} \mid i\geq 0 \right\}$ of tame open sets in $K$ with $V_{i} \subset V_{i+1}$ and $V=\cup_{i\geq 0} V_{i}$ such that $q \geq q_{V_{i}}$ for all $i\geq 0$, where $q_{V_{i}}$ is the handle index of $V_{i}$. \\
Reminder (compare section \ref{handleif}): The handle index of $V_{i}$ was defined by $q_{V_{i}}:= \text{max}_{j\in\N} \hspace{0,1cm} f_{V_{i}} (j)$ where $f_{V_{i}}$ is the handle index function of $V_{i}$.
\end{defi}

\begin{ex}
Let $V\in\mathcal{O}$ be a tame set. Then the homotopy dimension $\hodim(V)$ of $V$ equals the handle index $q_{V}$ of $V$. 
\end{ex}

\begin{thm}\label{analytic}
Let $F$ be a $\rho$-analytic functor with excess $c$. Let $V\in\mathcal{O}$ be an open subset with $ \hodim(V) =: q < \rho$. Then the map
\begin{align*}
\eta_{k-1}(V):F(V)\rightarrow T_{k-1}F(V)
\end{align*}
is $(c+k(\rho - q))$-connected for every $k>1$.
\end{thm}
\begin{proof}
Since the functor $F$ is good, we only have to consider the case where $V$ is a tame open subset. \\
We induct on the following statement depending on $j$: The map $\eta_{k-1}(V)$ is $(c+k(\rho - q))$-connected for all tame open sets $V\in\mathcal{O}$ with $f_{V}(m)\leq 0$ for all $m>j$. Here $f_{V}:\N\rightarrow \N\cup\left\{-\infty\right\}$ is the handle index function of $V$. \\
If $j=0$, the proof is essentially the same as in \cite[Theorem 2.3]{mancal2}:
By definition, $V$ is an element of $\mathcal{O}l$ where $l$ is the number of the components $V_{1},...,V_{l}$ of $V$. If $l<k$, then $V$ is a terminal object in $\mathcal{O}(k-1)(V)$ and thus $\eta_{k-1}$ is a weak equivalence. \\
Now we assume $l\geq k$. 
For $T\subset\left\{ 1,...,l \right\}$, we define $V_{T}:= \cup_{i\in T} V_{i}$. 
For a positive integer $t\leq l$, let $\mathcal{Z}t$ be the full subcategory of $\mathcal{O}t$ where the objects are all $V_{T}$ with $\left| T \right| \leq t$. Then there is a commutative square of inclusions of subcategories
\begin{equation*}
\begin{gathered}
\xymatrix{
\mathcal{Z}(t-1) \ar[d] \ar[r] & \mathcal{Z}t \ar[d] \\
\mathcal{O}(t-1) \ar[r] & \mathcal{O}t
}
\end{gathered}
\end{equation*}

If we set $J_{t}(V):= \holimsub{U\in\mathcal{Z}t} \hspace{0,1cm} F(U)$, we obtain a commutative square of spaces

\begin{equation*}
\begin{gathered}
\xymatrix{
T_{t}F(V) \ar[d] \ar[r]^{r_{t}} & T_{t-1}F(V) \ar[d] \\
J_{t}(V) \ar[r] & J_{t-1}(V)
}
\end{gathered}
\end{equation*}

The vertical arrows are weak equivalences because the category $\mathcal{Z}t$ is a homotopy terminal subcategory of the category $\mathcal{O}t$. \\
In order to show that the bottom horizontal arrow is a weak equivalence, we consider the following pullback square 
\begin{equation*}
\begin{gathered}
\xymatrix{
J_{t}(V) \ar[d] \ar[r] & J_{t-1}(V) \ar[d] \\
\prod_{\left\{T\subset[l]\mid t= \left|T\right| \right\}} \text{holim}_{R\subset T} F(V_{R}) \ar[r] & \prod_{\left\{T\subset[l]\mid t=\left|T\right|\right\} } \text{holim}_{R\subset T,R\neq T} F(V_{R})
}
\end{gathered}
\end{equation*}
where the vertical maps are the canonical maps and the horizontal map in the bottom row is induced by the canonical maps
\begin{align*}
p_{T}: \holimsub{R\subset T} F(V_{R}) \rightarrow \holimsub{R\subset T,R\neq T} F(V_{R})
\end{align*}
for all $T\subset[l]$ with $\left|T\right| = t$. We observe that the horizontal arrows are fibrations since the maps are canonical projection maps.
Now we use the analyticity assumption to verify that the map $p_{T}$ is $(c+t\rho)$-connected for every $\left|T\right| = t$. Using Theorem \ref{alongfib}, it follows that the map $J_{t}(V)\rightarrow J_{t-1}(V)$ is also $(c+t\rho)$-connected.  
If we summarize the previous results, we conclude that the composition
\begin{align*}
T_{l}F(V) \stackrel{r_{l}}{\longrightarrow} T_{l-1}F(V) \stackrel{r_{l-1}}{\longrightarrow} ... \stackrel{r_{k}}{\longrightarrow} T_{k-1}F(V)
\end{align*}
is $(c+k\rho)$-connected. Since the map $\eta_{l}(V)$ is a weak equivalence, the map 
\begin{align*}
\eta_{k-1}(V)=(\eta_{l}\circ r_{l} \circ r_{l-1} \circ ... \circ r_{k})(V):F(V)\rightarrow T_{k-1}F(V)
\end{align*}
is also $(c+k\rho)$-connected. \\

Now assume that the statements $0,1,...,j-1$ are proven. We have to verify statement $j$. 
We suppose that $f_{V}(j)=q$ for an integer $q>0$ and $f_{V}(m)\leq 0$ for all $m>j$. \\
Since $V$ is tame, there is a compact codimension zero subobject $C$ such that $\text{int}(C)=V$. For all simplices $S$ of $K$, we choose a handle decomposition of the compact codimension zero manifold $C\cap S$. \\ 
For all handles $Q_{u}$ of index $q$ and dimension $j$, choose a diffeomorphism $e: D^{q}\times D^{j-q} \rightarrow Q_{u}\subset C\cap K^{j}$ such that $e^{-1}(\partial (C \cap K^{j})) = D^{q}\times \partial D^{j-q}$. Since $q>0$, there are pairwise disjoint closed $q$-disks $B_{0}^{u},...,B_{k-1}^{u}$ in $D^{q}$. For $i\in [k-1]$, we set
\begin{align*}
A_{i}^{u}:= e(B_{i}^{u}\times D^{j-q})\cap V 
\end{align*}
Define $A_{i}$ to be the union of all collars $\text{col}_{V}(A_{i}^{u})$ of $A_{i}^{u}$ in $V$ (see Definition \ref{col}) for arbitrary $u$.  
By definition, $A_{i}$ is a closed subset of $V$ for each $i$. If we set $V_{T}:=V\setminus\cup_{i\in T}A_{i}$ for $\emptyset\neq T\subset[k-1]$, then $V_{T}$ is a tame open set with $f_{V_{T}}(j)<q$ and $f_{V_{T}}(m)\leq 0$ for all $m>j$.
We consider the following commutative square:
\begin{equation*}
\begin{gathered}
\xymatrix{
F(V) \ar[d] \ar[r] & \text{holim}_{\emptyset\neq T\subset [k-1]} F(V_{T}) \ar[d] \\
T_{k-1}F(V) \ar[r] & \text{holim}_{\emptyset\neq T\subset [k-1]} T_{k-1}F(V_{T})
}
\end{gathered}
\end{equation*}
We supposed that $F$ is $\rho$-analytic with excess $c$. Therefore, the map
\begin{align*}
F(V) \rightarrow \holimsub{\emptyset \neq T \subset [k-1]} F(V_{T})
\end{align*}
is $c+k(\rho - q)$-connected because the relative handle index of $A_{i}$ is $q$ (relative to the closure of $V_{[k-1]}$). 
By the induction hypothesis, we deduce that the map $F(V_{T})\rightarrow T_{k-1}F(V_{T})$ is $c+k(\rho-(q-1))$-connected. By \cite[1.22]{goodwillie2}, the induced map
\begin{align*}
\holimsub{\emptyset \neq T \subset [k-1]}F(V_{T}) \rightarrow \holimsub{\emptyset \neq T \subset [k-1]} T_{k-1}F(V_{T})
\end{align*}
is $(c+k(\rho-q-1)-k+1)$-connected. 
Since $T_{k-1}F$ is $(k-1)$-polynomial, the map 
\begin{align*}
T_{k-1}F(V) \rightarrow \holimsub{\emptyset \neq T \subset [k-1]} T_{k-1}F(V_{T})
\end{align*}
is a weak equivalence. We have proven that the map $F(V)\rightarrow T_{k-1}F(V)$ is $c+k(\rho - q)$-connected. 
\end{proof}

\begin{rmk}\label{bettercondition}
In the definition of analyticity there appear codimension zero subobjects $P$ and $A_{i}$, $0\leq i\leq r$. We could impose stronger conditions on these subobjects which would weaken the definition of analyticity, but the last theorem would still hold. What are these conditions? To answer this question we have to ask where we used the analyticity assumption in the proof of the last theorem. We used it twice and we can summarize that we can assume that the relative handle index functions $f_{A_{i}}$ (relative to $P$), $0\leq i\leq r$, have one of the following two forms:
\begin{enumerate}
    \item We can assume that $P$ is empty and $f_{A_{i}}(m)\leq 0$ for all $m\in\N$ and $i\in[r]$.
    \item There exists $j\in\N$ such that $f_{A_{i}}(j)=q$ and $f_{A_{i}}(m)=-\infty$ for all $m\neq j$ and $i\in[r]$. In addition, $f_{\text{int}(P)}(m)=-\infty$ for all $m> j$ where $f_{\text{int}(P)}$ is the handle index function of $\text{int}(P)$ - the interior of $P$.
\end{enumerate}
Therefore, we could assume that the codimension zero subobjects in the definition of analyticity (Definition \ref{nalytic}) fulfil either (i) or (ii). We get a weaker condition for analyticity, but Theorem \ref{analytic} would still hold. 
\end{rmk}

\begin{cor}
Let $F$ be a $\rho$-analytic functor with $\rho > \text{dim}(K)$. For all open sets $V\in\mathcal{O}(K)$, the canonical map
\begin{align*}
F(V)\rightarrow T_{\infty}F(V) = \holim_{k} T_{k}F(V)
\end{align*}
is a weak equivalence.
\end{cor}

\section{Examples}\label{applications}




Now we consider first applications of the theory which we developed in this paper.

\subsection{Spaces of embeddings}

Let $N$ be a smooth manifold without boundary such that $\text{dim}(K)\leq \text{dim}(N)$ and let $V$ be an open subset of $K$. We define the space $emb(V,N)$ to be the space of topological embeddings $e:V\rightarrow N$ such that $e|_{S\cap V}:(S\cap V)\rightarrow N$ is a smooth embedding for all simplices $S$ of $K$. 
Now we can introduce the contravariant functor 
\[
emb(-,N) : \mathcal{O}(K) \rightarrow (Top)
\]
by $V\mapsto emb (V,N)$. The verification of goodness (in the sense of Definition \ref{good}) is an easy exercise which is left to the reader. It is similar to its analogue in the setting where $K$ is replaced by a smooth manifold \cite[Proposition 1.4]{mancal1}.

\begin{thm}\label{spofemb}
If $\text{dim}(K)+3\leq \text{dim}(N)$, then $emb(-,N)$ is analytic (i.e. it fulfils the condition in Remark \ref{bettercondition}).
\end{thm}
\begin{proof}
We will not give all details of the proof since many of them are equal to the arguments of \cite[1.4]{mancal2}. Let $P$ be a codimension zero subobject of $K$ and let $A_{0},...,A_{r}$ be pairwise disjoint codimension zero subobjects of $K\setminus \text{int}(P)$ fulfilling the following conditions: 
For each $i\in[r]$\hspace{0,5mm}, let $f_{A_{i}}:\N\rightarrow\N$ be the relative (to $P$) handle index function. We assume that there exists a $j\in\N$ such that $f_{A_{i}}(m)=-\infty$ for all $m\neq j$ and $i\in[r]$. (In addition, we can desire that $f_{\text{int}(P)}(m)=-\infty$ for all $m> j$ where $f_{\text{int}(P)}$ is the handle index function of $\text{int}(P)$.) 
For $T\subset[r]$\hspace{0,5mm}, we set $A_{T}:=\cup_{i\in T} A_{i}$ and $V_{T}:=\text{int}(A_{T}\cup P)$. \\
We start with the following observation: By definition, the restriction map
\begin{align}\label{verdickung}
emb(int(A_{i}),N) \rightarrow emb(int(A_{i})\cap K^{j},N)
\end{align}
is a weak equivalence. Here $emb(V\cap K^{j},N)$ is a subspace of $emb( K^{j},N)$ for all $V\in\mathcal{O}(K)$. \\
Similarly to the proof in the case of manifold calculus \cite[1.4]{mancal2}, we have to show that the $k$-cube
\[T \mapsto (ho)fiber \left[ emb(cl(V_{T})\cap K^{j},N)\rightarrow emb(cl(V_{\emptyset})\cap K^{j},N) \right]
\]
is $(3-n+\sum_{i=0}^{r} (n-q_{A_{i}}-2))$-cartesian. Here $cl(V_{T})$ is the closure of $V_{T}$ in $K^{j}$ and $emb(cl(V_{T}),N)$ is the homotopy limit of $emb(U,N)$ where the homotopy limit ranges over all neighbourhoods $U$ of $cl(V_{T})$ in $K^{j}$. \\
Why is it enough to show that this cube is highly cartesian? First of all, we observe that the restriction map from $emb(cl(V_{T}),N)$ to $emb(V_{T},N)$ is a weak equivalence (since $V_{T}$ is a tame open subset of $K$). In addition, we observe that the restriction maps from $emb(cl(V_{T}),N)$ to $emb(cl(V_{\emptyset}),N)$ are fibrations. This follows from the Isotopy Extension Theorem for manifolds which can be applied because of the special assumptions on the codimension zero subobjects $A_{i}$ where $i\in[k]$. Then we can use \cite[Lemma 1.2]{mancal2} and the weak equivalence given in (\ref{verdickung}). \\
Why is the cube highly cartesian? We define $D(cl(V_{\emptyset}))$ to be the normal disc bundle for $cl(V_{\emptyset})$ in $N$. This is the union of the normal disc bundles of $cl(V_{\emptyset})\cap S$ for all simplices $S$ of $K$. They have to be compatible in the following sense: $D(cl(V_{\emptyset}))$ is a smooth codimension zero submanifold of $N$ with corners. \\
We set $Y$ as the closure of $N\setminus D(cl(V_{\emptyset}))$ in $N$, then $Y$ is a manifold with boundary. Since for every $i\in[k]$, $A_{i}\cap K^{j}$ is a $j$-dimensional manifold by assumption, we are exactly in the situation of proof \cite[1.4]{mancal2}. Now we can proceed with the same arguments, in particular we can apply \cite[1.3]{mancal2}. 
\end{proof}

Now let $L$ be another simplicial complex. Let $\mathcal{S}(K)$ be the set of all simplices of $K$ and let $\mathcal{S}(L)$ be the set of all simplices of $L$. Let $\Psi : \mathcal{S}(K)\rightarrow\mathcal{S}(L)$ be a map of sets. Then we define $emb_{\Psi}(K,L)$ to be the space of all topological embeddings $f:K\rightarrow L$ such that for every simplex $S$ of $K$, the restricted map $f|_{S}$ takes $S$ to $\Psi (S)$ and $f|_{S}$ is a smooth embedding of manifolds with $f|_{S}^{-1}(\partial \Psi(S)) \subset \partial S$. Note: In many cases this space will be empty because the choice of $\Psi$ does not always allow continuous maps $K\rightarrow L$ with these additional properties. \\
More generally, let $V\in\mathcal{O}(K)$ be an open subset of $K$. Then we define $emb_{\Psi}(V,L)$ to be the space of all topological embeddings $f:V\rightarrow L$ such that for every simplex $S$ of $K$, the restricted map $f|_{S\cap V}$ takes $S\cap V$ to $\Psi(S)$ and $f|_{S\cap V}:S\cap V\rightarrow \Psi(S)$ is a smooth embedding of manifolds with $f|_{S\cap V}^{-1}(\partial \Psi(S)) \subset \partial S\cap V$. \\
There is a contravariant functor 
\[
emb_{\Psi}(-,L) : \mathcal{O}(K) \rightarrow (Top)
\]
given by $V\mapsto emb_{\Psi}(V,L)$. The following theorem can be proven in the same way.

\begin{thm}
If $\text{dim}(\Psi(S))-\text{dim}(S)\geq 3$ for all simplices $S$ of $K$, 
the functor $emb_{\Psi}(-,L)$ is analytic (i.e. it fulfils the condition in Remark \ref{bettercondition}).
\end{thm}

\subsection{Occupants in simplicial complexes}\label{occsection}

Let $M$ be a smooth manifold without boundary and let $K$ be a subset of $M$. We can ask: Is it possible to recover the homotopy type of $M\setminus K$ from the homotopy types of the spaces $M\setminus T$ where $T$ is a finite subset of $K$? In some cases it is possible if we allow thickenings of the finite subsets $T$ and allow inclusions between them. 

In a joint paper with Michael Weiss \cite{occupants}, we investigated the case where $L$ is a submanifold of a Riemannian manifold $M$ (also with empty boundary) of codimension $\geq 3$. Let $con(L)$ be the configuration category of $L$. The objects of $con(L)$ are pairs $(T,\rho)$ where $T$ is a finite subset of $L$ and $\rho: T\rightarrow (0,\infty)$ is a function which assigns to each element $t\in T$ the radius $\rho(t)$ of the corresponding thickening. These pairs have to fulfil different conditions (for a precise definition, see \cite{occupants}). For each object $(T,\rho)$ in $con(L)$, there exists a corresponding open subset $V_{L}(T,\rho) \subset L$ which is a canonical thickening of the finite subset $T\subset L$. It is a union of the open balls of radius $\rho(t)$ about the points $t\in T$. 
For each element $(T,\rho)$ of the configuration category, we get an inclusion
\[
M\setminus L \rightarrow M\setminus V_{L}(T,\rho)
\]
The main result of \cite{occupants} is the following theorem:

\begin{thm}\label{th1}
In these circumstances, the canonical map
\[
M\setminus L \rightarrow \holimsub{(T,\rho)\in con(L)} M\setminus V_{L}(T,\rho)
\]
is a weak equivalence.
\end{thm}

The paper also includes many variants of this result, e.g. a variant with restricted cardinalities and we considered manifolds with boundaries and corners. I would like to emphasize the following variant: Let $M$ be a manifold with boundary $\partial M$. Then we want to recover the homotopy type of $\partial M$ from the homotopy types of the spaces $M\setminus T$ where $T$ is a finite subset of $M\setminus \partial M$. Again, we need to allow thickenings of the finite subsets $T$ and inclusions between them. Therefore, we consider the configuration category $con(M\setminus \partial M)$ of the interior of $M$. For each object $(T,\rho)$ in $con(M\setminus \partial M)$, there is a corresponding open set $V(T,\rho)$ defined in the known way and an inclusion
\[
\partial M \rightarrow M\setminus V(T,\rho)
\]

\begin{thm}\label{th2}
The canonical map 
\[
\partial M \rightarrow \holimsub{(T,\rho)\in con(M\setminus\partial M)} M\setminus V(T,\rho)
\]
is a weak equivalence if the following condition holds: There exists a smooth disc bundle $M\rightarrow L$ over a smooth closed manifold $L$ with fibers of dimension $c\geq 3$.
\end{thm}

Now let $K\subset M$ be a simplical complex such that $S$ is smoothly embedded in $M$ for each (closed) simplex $S$ of $K$. We do not go into detail, but there is also a category of canonical thickenings of finite subsets of $K$ - denoted by $con(K)$. The objects of $con(K)$ are again pairs $(T,\rho)$ where $T$ is a finite subset of $K$ and $\rho: T\rightarrow (0,\infty)$ is a function such that some expected conditions hold. We have again corresponding open subsets $V_{K}(T,\rho)$ and inclusion $M\setminus K\rightarrow M\setminus V_{K}(T,\rho)$. In my paper \cite{occupantsinsimplicialcomplexes}, I prove the following generalization of Theorem \ref{th1}:
\begin{thm}\label{th3}
If the codimension of $K$ and $M$ is at least three, the canonical map
\[
M\setminus K \rightarrow \holimsub{(T,\rho)\in con(K)} M\setminus V_{K}(T,\rho)
\]
is a weak equivalence.
\end{thm}
We can use this theorem to weaken the conditions in Theorem \ref{th2}. The canonical map in \ref{th2} is a weak equivalence if $M$ is a regular neighbourhood of a compact simplicial complex of codimension $c\geq 3$.

\section{Appendix}

\subsection{Theorems for the homotopy (co-)limit}

Let $\mathcal{S}$ be the category of topological spaces or simplicial sets. The following two propositions are proven in \cite[8.6]{mancal1}.

\begin{prop}\label{quasifibration}
Let $C$ be a small category and $F:C\rightarrow \mathcal{S}$ be a functor which takes all morphisms in $C$ to homotopy equivalences. Then the canonical map 
\begin{align*}
\hocolim_{C} F \hspace{0,15cm}\rightarrow \hspace{0,15cm}\left| C\right|
\end{align*}
is a quasifibration.
\end{prop}

\begin{prop}\label{holimassfib}
Let $C$ be a small category and $F:C\rightarrow \mathcal{S}$ be a functor which takes all morphisms in $C$ to homotopy equivalences. Then there is a homotopy equivalence between $\holim_{C} F$ and the section space of the associated fibration of the quasifibration $\hocolim_{C} F \rightarrow \left| C\right|$.
\end{prop}

\begin{prop}\label{homequsmcat}
Let $F:J\rightarrow \mathcal{S}$ be a functor which takes all morphisms to weak equivalences. If $i: I\rightarrow J$ is an inclusion of small categories such that $\left| I \right| \rightarrow \left| J \right|$ is a homotopy equivalence, then we have a weak equivalence 
\[
\hocolim_{I} F\circ i^{*}
\rightarrow
\hocolim_{J} F
\]
\end{prop}
\begin{proof}
Let $x$ be an element of $\left| I \right|$. Using the inclusion $\left| I \right| \rightarrow \left| J \right|$, we can also consider $x$ as an element of $\left| I \right|$. The fibers under the projection maps $\text{hocolim}_{J} F \rightarrow \left| J \right|$ and $\text{hocolim}_{I} F\circ i^{*} \rightarrow \left| I \right|$ of $x$ coincide. By Proposition \ref{quasifibration}, the homotopy fibers also coincide (up to homotopy). Then the assertion follows from the Five lemma (compare the long exact fiber sequences).
\end{proof}

\begin{thm}\label{alongfib}
Suppose we have a pullback square 
\begin{equation*}
\begin{gathered}
\xymatrix{
A \ar[d] \ar[r]^{f} & B \ar[d] \\
C \ar[r]^{g} & D
}
\end{gathered}
\end{equation*}
where $g$ is an $n$-connected Serre fibration. Then the map $f$ is also $n$-connected.
\end{thm}
\begin{proof}
Since $g$ is a fibration, the pullback square is also a homotopy pullback square \cite[13.3]{Hirschhorn}. Therefore, the map $f$ is also $n$-connected.
\end{proof}

\end{document}